\documentclass[
a4paper,
eqnright, reqno]{amsart}
\usepackage{graphicx}
\usepackage{hyperref}
\usepackage{color}
\definecolor{black}{gray}{0}
\definecolor{ttgray}{gray}{0.5}
\definecolor{bfred}{rgb}{0.4,0,0}
\definecolor{emgreen}{rgb}{0,0.3,0}
\definecolor{emmagenta}{rgb}{0.6,0.07,0.07}
\definecolor{sfgb}{rgb}{0,0.3,0.3}
\definecolor{mathblue}{rgb}{0,0,0.4}
\usepackage{myamsart}
\usepackage[opta]{optional}

\providecommand{\sinh}{\loglike{sinh}}
\providecommand{\cosh}{\loglike{cosh}}

\newtheorem{ex}{Example}
\title[Stability Analysis of Matrix W--H Factorisation of
  Daniele--Khrapkov]{Stability Analysis of Matrix Wiener--Hopf Factorisation of
  Daniele--Khrapkov Class and Reliable Approximate Factorisation}
\author{Anastasia V. Kisil}    
\address{Cambridge Centre of Analysis, University of Cambridge,
  Wilberforce Road, Cambridge, CB3 0WA, UK }
\email{a.kisil@maths.cam.ac.uk}
 
\begin{document}
 
\begin{abstract}

This paper presents new stability results for matrix Wiener--Hopf
factorisation. The first part of the paper examines  conditions for
stability of Wiener-Hopf factorisation in
 Daniele--Khrapkov class. 
The second part of the paper concerns the class of matrix functions
which can be exactly or approximately reduced to the
factorisation of the Daniele--Khrapkov matrices. The results
of the paper are demonstrated by numerical examples with partial indices
\(\{1,-1\}\), \(\{0,0\}\) and \(\{-1,-1\}\).
\end{abstract} 


\maketitle


\section{Introduction}

This paper examines the stability of Wiener--Hopf matrix
factorisation~\cites{Spit,Constructive_review,Spit_book} in a certain class
of matrices. In the essence, a factorisation of a scalar or matrix
function \(\mathbf{G}( t)\) is its decomposition into a product
\begin{equation}
  \label{eq:wiener-hopf-defn}
  \mathbf{G}(  t)=\mathbf{G}_+(t)\mathbf{G}_-(t)  
\end{equation}
with the invertible factors
\(\mathbf{G}_+(t)\) and \(\mathbf{G}_-(t)\) analytically extendable
into the upper/lower half-plane (Section~\ref{subsec:W-H}). We
consider the class of Daniele--Khrapkov \(2\times 2\)
matrices~\cites{daniele2014wiener,Khrap}, which have the form
\begin{equation}
  \label{eq:daniele-kharpkov-defn}
  \mathbf{K}(t)= \mathbf{I}+f(t) \mathbf{J}(t),  
\end{equation}
where \(f(t) \) is an arbitrary scalar function with algebraic
growth at infinity, and \(\mathbf{J}(t)\) is a polynomial matrix with
\[ \mathbf{J}^{2}(t)=\Delta^{2}(t) \mathbf{I},\]
where \(\Delta^{2}(t)\) is a polynomial in \(t\) and \( \mathbf{I}\)
is the \(2\times 2\) identity matrix. The  Daniele--Khrapkov matrices
can be factorised explicitly (Section~\ref{sec:error}).

Fundamentally, the scalar and the matrix Wiener--Hopf
factorisations~\eqref{eq:wiener-hopf-defn} are different: the former
has a constructive solution in terms of a Cauchy type integral and the
latter has no explicit solution in general. The existence of matrix
factorisation under general assumptions has been proved by Gohberg and
Krein~\cite{existance}. Nevertheless, up to date constructive matrix
factorisation remains a formidable challenge. Due to its complexity,
different classes of matrix functions have to be treated separately
(see the recent review article~\cite{Constructive_review}). The class
of Daniele--Khrapkov matrices is very important for applications and
arise naturally in a number of interesting problems in acoustic,
electromagnetics, etc., see for
example~\cites{Elastic_half,Ab_Kh_n,daniele2014wiener}.

The Wiener--Hopf factorisation~\eqref{eq:wiener-hopf-defn} is said to
be \emph{stable} if small changes in the matrix function \(\mathbf{G}(
t)\) lead to small changes in the factors \(\mathbf{G}_+(t)\) and
\(\mathbf{G}_-(t)\) (Section~\ref{sec:error}). Almost all
implementations of Wiener--Hopf technique are performed
numerically~\cite{history}, therefore  a careful analysis of stability is
essential. Among popular approximate techniques are truncated pole
removal~\cite{Nigel_poles} and rational
approximations~\cites{Nigel_cyl,Abraha_all_pade}. There are also new
asymptotic methods
~\cites{Mishuris-Rogosin,assym_2,Crighton_matching}, which also rely
on stability.  Even in the rare cases when explicit factorisations are
known, e.g. for Daniele--Khrapkov matrices, they still require numerical
computations of scalar factorisations. Those computations introduce
small errors, which can lead to large errors in the Wiener--Hopf
factors (Section~\ref{sec:Stability}).
 
A landmark theorem of Gohberg and Krein~\cite{Spit_book}*{\S~6.2}
gives general conditions for stability of matrix Wiener--Hopf
factorisation (Section~\ref{sec:Stability}). The difficulty of
applying these results is that the stability conditions depend on the
knowledge of Wiener--Hopf factorisation and hence are impractical to
check. The aim of this paper is to provide direct criteria for
stability of factorisation in a case of Daniele--Khrapkov
matrices. The conditions are demonstrated by numerical examples.

This work is a continuation of the author's paper~\cite{My1}, which
demonstrated a novel method of approximately solving scalar
Wiener--Hopf equations. In the scalar case the formula for the
solution in terms of a Cauchy type integral was used to bound the
error in the factors.  In this paper the previous results are extended
to the Daniele--Khrapkov matrices.

The first part of the paper establish stability of the
Daniele--Khrapkov class under perturbations within the class. There
are benefits to considering the `near' matrices  only within the
class.  It allows to answer the question if numerical
implementation of the factorisation is stable. This also allowed to 
obtain explicit error bounds.  The third advantage is that in a
specific case stronger results can be obtained then in the general
case.


The second part of the paper extends the class of matrix functions to
these which can be approximately reduced to  Daniele--Khrapkov
matrices. The class of matrix functions considered by Abrahams
in~\cite{Ab_noncom} is a special case of this construction. It is
shown that the stability results could be applied to this  meromorphic
factorisation. This is then used to show stability in a interesting 
example.

\opt{optx}{\section{Motivation}
\label{sec:Motivation}

This section provides motivation for the Wiener--Hopf factorisation,
the central concept in this paper.  The first subsection describes a
basic procedure to solve Wiener--Hopf equations. This highlights the use of
Wiener--Hopf factorisation and additive splitting as well as
demonstrating the necessity of algebraic behaviour at infinity. The
second subsection gives examples of an integral equation and
a differential equation which reduce to Wiener--Hopf equations. 

\subsection{The Wiener--Hopf Method}

The Wiener--Hopf problem is to find unknown functions \(\Phi_+\) and
\(\Psi_-\) which satisfy the following equation on the real line:
\[A( \alpha)\Phi_+(\alpha)+ \Psi_-(\alpha)+C(\alpha)=0.\] 
Functions \(A( \alpha)\) and \(C(\alpha)\) are given.
The subscript \(+\) (or \(-\)) indicates  that the
function is \emph{analytic  in the upper (or lower) half-plane}.

The key  step in the Wiener Hopf procedure is to find invertible
functions \(K_+\)
and \(K_-\) (\emph{the Wiener--Hopf factors}) such that:
\[A(\alpha)=K_+(\alpha)K_-(\alpha).\]
Then multiplying through we have:
\[K_+ \Phi_+ ( \alpha)+ K_-^{-1}  \Psi_-(\alpha)+ K_-^{-1}C(\alpha)=0.\]
Now splitting additively we obtain:
\[ K_-^{-1}C(\alpha)=C_-(\alpha)+C_+(\alpha).\]
Finally rearranging, the equation becomes:
\[K_+ \Phi_+ ( \alpha) +C_+( \alpha)=-K_-^{-1}
\Psi_-(\alpha)-C_-(\alpha).\]
The right hand side is a function that is analytic in the upper half-plane and the real line
and the left is analytic in the lower half-plane and the real line. So by analytic continuation 
obtain a function \(J(\alpha)\) that is analytic everywhere in the complex
plane. The requirement on the equation is that there exists an \(n\) such that (\emph{at most polynomial growth}):
\[J(\alpha)<|\alpha|^{n} \quad |\alpha| \to
\infty.\]
Applying the extended Liouville's theorem to \(J(\alpha)\) implies
that \(J(\alpha)\) is
polynomial of degree less than or equal to \(n\). So the solution is reduced
 to finding \(n\) unknown constants which can be
determined in some other way (typically \(n=0\) or \(1\)). Hence  \(\Phi_+ (
\alpha)\) and \(\Psi_-(\alpha)\) are determined.

\begin{rem}
In the matrix Wiener--Hopf the same procedure follows through. The
difficulty is contained in the finding the Wiener--Hopf factors.  
\end{rem}

\subsection{Applications}

Numerous applications lead to integral equations of the form 
\[ \int_0^{\infty} k(x-y)f(y) dy=g(x) \quad 0<x< \infty,\]
where \(k(x-y)\) is a given kernel and \(g(x)\) is also given (with
\(f(y)\) to be determined). Extend the domain of \(x\)
\[  \int_0^{\infty} k(x-y)f(y) dy=h(x) \quad -\infty<x< 0,\]
where \(h(x)\) is unknown. Then after applying the Fourier transform 
we obtain the Wiener--Hopf equation:
\[ F_+ (\alpha)K(\alpha)-G_+(\alpha)= H_-(\alpha).\]
This Wiener--Hopf equations can then be solved for \(F_+\) and mapped back by
taking inverse Fourier transform.

Another typical example of how the Wiener--Hopf equation appears is from a  PDE in acoustics.
Consider the boundary value problem  governed by the
reduced wave equation:
\[\left(\frac{\partial}{\partial x^2}+\frac{\partial}{\partial y^2}
+c^2 \right)
\phi (y, x) =0,  \quad \text{with} \quad c \in \mathbb{R},\]
and in the semi-infinite region \(-\infty
<x< \infty \) and \(y\ge 0\), with boundary conditions:
\begin{eqnarray}
  \label{eq:f}
  \phi &=& f(x) \quad \text{on} \quad y=0 \qquad 0 <x< \infty, \\
\frac{\partial  \phi}{\partial y} &=& g(x) \quad \text{on} \quad y=0 \quad -\infty <x< 0.
\end{eqnarray}
Applying the Fourier transform to the equation and the boundary
conditions yield the Wiener--Hopf equation see~\cite{bookWH} for more
details.
}
\section{Preliminaries}


Throughout the paper we are using the subscripts \(+\) and \(-\) to
denote functions which admit an analytic continuation into the the
upper and lower half-planes respectively. The \emph{Wiener algebra}
\(W(\mathbb{R})\) over the real line~\cite{Spit}*{Ex.~2.2} consists of
all complex valued functions \(f\) in \(\mathbb{R}\) that admit a
representation of the form 
\begin{displaymath}
  f(\lambda)=d+\int\limits_{-\infty}^{\infty} e^{i\lambda t}
  k(t)\,dt\, ,\quad \lambda\in\mathbb{R},
\end{displaymath}
for some  \(d\in\mathbb{C}\) and \(k\in L_1(\mathbb{R})\).

\subsection{Wiener--Hopf factorisation}
\label{subsec:W-H}

This subsection recalls the different types of Wiener--Hopf
factorisation, which have their own merits,
see~\cite{Spit} for a detailed exposition. 
Let \(\mathbf{G}( t)\) be in the matrix Wiener algebra \(W_{2 \times
  2}(\mathbb{R})\)~\cite{Spit_book}*{\S~5.2}. If \(\text{det }
\mathbf{G}( t) \neq 0\) for all real \(t\) then there exists the
full factorisation
\begin{equation}
\label{eq:RH}
  \mathbf{G}( t)=\mathbf{G}_+(t)\mathbf{D}(t)\mathbf{G}_-(t), \qquad t \in \mathbb{R},
\end{equation}
where factors and their inverses belong to the subalgebras of
analytically extendable functions to the respective half-planes
\begin{align*}
  \mathbf{G}_+^{\pm 1} &\in W_{2 \times 2}^+(\mathbb{R}), \quad \mathbf{G}_-^{\pm 1} \in W_{2 \times 2}^-(\mathbb{R}), \\
\text{ and } \mathbf{D}(t)&=\text{diag}\Big[\left(\frac{t-i}{t+i}\right)^{\kappa_1}, \left(\frac{t-i}{t+i}\right)^{\kappa_2}\Big].
\end{align*}
The integer exponents \(\kappa_1\) and \(\kappa_2\) are called
\emph{partial indices}. Unlike factorisation, the partial indices are
unique.  But in contrast to the scalar case, they cannot be determined
\emph{a priori} in general. 

A factorisation~\eqref{eq:wiener-hopf-defn} with the invertible
factors \(\mathbf{G}_+(t)\) and \(\mathbf{G}_-(t)\)
analytically extendable into the respective half-planes and
polynomilally bounded growth at infinity will be called
\emph{function-theoretic} factorisation.  The function-theoretic
factorisation is useful in applications since it retains most
information and is easier to find. 

\begin{rem}
  The partial indices are linked to the growth at infinity in
  function-theoretic factorisation, see~\cite{Mer}.
\end{rem}

It is also useful to consider a \emph{meromorphic} factorisation, where
the conditions are further relaxed to allow the presence of a finite
number of poles and zeroes in the
factors.

\opt{optx}{\subsection{Non-Uniqueness of Factorisation} 
\label{subsec:N-U}

This section details the degree of freedom when constructing
Wiener--Hopf factorisations.  In the scalar case, factors are unique up to
a constant. In other words if there are two such factorisations
\(K=K_+K_-\) and \(K=P_+P_-\) then:
\[K_+=cP_+ \quad \text{and} \quad K_-=c^{-1}P_-,\]
where \(c\) is some complex constant. This can be seen by
applying analytic continuation to \(K_+/P_+\) and \(P_-/K_-\) and then using
the extended Liouville theorem~\cite{Kranzer_68}.

The corresponding result in the matrix case is more complicated:

\begin{thm}~\cite{Spit}
\label{thm:Uni}
Let the matrix \(\mathbf{G}( t)\) admit the Wiener--Hopf factorisation (\ref{eq:RH}). Then: 
\[\mathbf{G}=\mathbf{G}_+^{1}\mathbf{D}\mathbf{G}_-^{1}, \]
  where
  \[\mathbf{G}_+^{1}=\mathbf{G}_+\mathbf{H}, \quad \mathbf{G}_-^{1}=\mathbf{D}^{-1}\mathbf{H}^{-1}\mathbf{D}\mathbf{G}_-,\]
  is also a factorisation of \(\mathbf{G}( t)\) where \(\mathbf{H}\) is a constant invertible 
  matrix function  if \(\kappa_1=\kappa_2\) and otherwise is of the form: 
  \[ \mathbf{H}(t)=\left(
 \begin{array}{cc}
  
c_1 &  P(z) \\
0  & c_2 
 \end{array} \right),\]
where \(P(z)\) is a polynomial of degree   \(\kappa_1-\kappa_2\) in \(z=\frac{t-i}{t+i}\).
\end{thm}	

}

\subsection{Scalar Error Estimates}
\label{subsec:SEE}

The \emph{index}  of a continuous non-zero
function \(K(t)\) on the real line 
is:
\begin{equation}
  \label{eq:eq:index-func-defn}
  \mathrm{ind}(K(t))=\frac{1}{2\pi}\big(
  \lim_{t\to+\infty}\arg  K(t)-\lim_{t\to-\infty}\arg
  K(t)\big).
\end{equation}
Note that
\(\mathrm{ind} \frac{t-i}{t+i}=1.\)
Thus, given a function \(K(t)\)  with index \(\kappa\) one can reduce it to zero index
by considering  \[K(t)\left( \frac{t-i}{t+1}\right) ^{-\kappa}.\] For the
rest of this subsection it will be assumed that all functions have
zero index. 

We also assume that \(K(t) \to 1\) for \(t \to \pm \infty\), then we
can normalise factors such that \(K_{\pm}(t) \to 1\) for \(t \to \pm \infty\).
A  non-zero H\"{o}lder continuous function \(K(t)\) on the real line with \(K(t)-1\)
in \(L_2(\mathbb{R})\)  possesses a
factorisation~\cite{Ga-Che}  \[K(t)=K_+(t)K_-(t),\] where \(K_{\pm}(t)\) are limiting
values of functions analytic and non-zero in the respective half-planes. 

The distinctive feature  of the scalar factorisation is the ability to
express the factors in terms of the Cauchy type integrals.
It is the existence of such expressions and the bounds in \(L_p\) on the Hilbert
transform which allowed to obtain some useful estimation~\cite{My1}.
We adapt them here for \(L_2\) case in the following form.

\begin{thm} [Additive Estimates in \(L_2\)]
\label{thm:add}
Let \(F(t)= F_+(t)+  F_-(t)\) and \(\tilde{F}(t)=
\bar{F}_+(t)+  \tilde{F}_-(t)\) with
\(\|F(t)-\tilde{F}(t) \|_2 < \epsilon\) then
\[\|F_{\pm}(t)-\tilde{F}_{\pm}(t) \|_2 \le \epsilon.\]
\end{thm}
\begin{thm}[Multiplicative Estimates in \(L_2\)]
\label{thm:mul}
Let  \(K(t)=K_{+}(t)K_{-}(t)\) and \(\tilde{K}(t)=\tilde{K}_{+}(t)\tilde{K}_{-}(t)\) be two
functions and \(m<|K|<M\). If
\(\|K(t)-\tilde{K}(t)\|_2< \epsilon\) then
\[\|K_{\pm}(t)-\tilde{K}_{\pm}(t)\|_{2} <
\frac{5(M+\epsilon)^{1/2}}{(m-\epsilon)} \epsilon.\] 
\end{thm}
The above results are special cases of theorems from~\cite{My1} with
some more explicit constants calculated.


\opt{optx}{\subsection{Constructive 
Factorisation of  Matrix Functions}
\label{subsec:class}

This section reviews the classes of matrix functions which have
constructive factorisation. By constructive factorisation it is meant
thet there is an algorithmic way to obtain the factorisation.
Constructive factorisation algorithms exist for classes of:

\begin{itemize}
\item Rational  matrix functions 
\item Functionally commuting matrix functions
\item Upper or lower triangular matrix function
\item Daniele--Khrapkov matrices
\end{itemize}

 What complicates the application of algorithms further is
that given a matrix it is difficulty in determining which class if any it
belongs to (pre and post multiplication can transform the matrix to a
right form but this is done using \emph{ad hoc} techniques). 

Rational matrix functions are a prime example where factorisation can
be achieved. This is because there are only poles and zeroes of
determinant which need to be considered i.e. meromorphic
factorisation is trivial. For a procedure see~\cite{Spit}. 

There has been some significant progress in the case of the matrix 
possessing a commutative factorisation~\cite{Spit_book}. One of
the reasons is that one could apply some of the
techniques that have been developed in the scalar
setting. Unfortunately, the class of matrices with a commutative
factorisation is not dense.
For some conditions under which the 
commutative factorisation is possible see~\cite{comm}. 

 Upper or lower triangular matrix functions (with factorable diagonal elements) can be reduced to a scalar
 equation and once this is solved this can be substituted to yield a
 second scalar equation. There are two systems one of which affects the behaviour of
 the other but not the other way round. So solving for the first
 system is a scalar problem and the effect of that system can be
 subtracted to leave the second system uncoupled. It is
 apparent why the general matrix Wiener--Hopf systems are so difficult
 to solve: there are two entangled systems which cannot be solved
 independently from each other.
 
 Daniele--Khrapkov matrices will be considered in detail in
 Section~\ref{sec:error}.
}
\opt{optx}{\subsection{Matrix Wiener--Hopf for Certain Class of Matrices}

 The first class of matrix presented here was considered by A.D. Rawling and W.E. Williams in~\cite{Raw}.
 It works for the matrices of the type:
 
 \[\mathbf{A} (\alpha) =\left(
 \begin{array}{cc}
  F(\gamma) &  G (\gamma)F(\gamma)  \\
 H(\gamma) & -G(\gamma)H(\gamma)
 \end{array} \right)\]

where \( F,G\) and $H$ are analytic functions (except possibly at
$\gamma=0$) functions of the variable $\gamma=(k^2- \alpha ^2)^{1/2}$ where $k$ is a constant with positive real and imaginary parts and the branch of the square root is chosen so that $\gamma=k$ when $\alpha=0$ with the branch cuts along the half-lines $ \alpha= k+ \delta$, $ \alpha= -k- \delta$ for $ \delta$ positive.

The way that this class of matrices is solved is by first transforming them to a matrix Riemann-Hilbert problem on a half line
and then it turns out that this system can be decoupled into scalar Riemann-Hilbert problem. Like the scalar Wiener--Hopf in the half line setting the resulting scalar Riemann-Hilbert can be solved explicitly.

The next class of matrices was considered by D. S. Jones in~\cite{Jones}:

\[\mathbf{C}  =\left(
 \begin{array}{cc}
  f &  ge_1  \\
 ge_2 & f- 2ge_3
 \end{array} \right)\]
 
 where $e_1$, $e_2$ and $e_3$ analytic function of complex variable $s$ but $f(s)$, $g(s)$ are analytic only in the strip.
 Then a commutative factorisation is found. Although it is not obvious but the class of matrices considered by 
A.D. Rawling and W.E. Williams can be reduced to Jones' class~\cite{Jones}. Then extensions to non-commutative factorisation are considered by pre-multiplying by  analytic matrices.
}

\section{Stability of Matrix Wiener--Hopf}
\label{sec:Stability}

For a sake of completeness we review here the most general results on
stability of matrix factorisation, since they are not widely known in
the Wiener--Hopf community.  The examples are adapted from a different
context of a Riemann-Hilbert problem on a circle.  There is a wealth
of different classes of factorisations considered by different
authors, for the purpose of clear exposition we consider here only
factorisation in Wiener algebra~\eqref{eq:RH}.

The simplest example of instability is obtained by mapping an example~\cite{Spit}
from the unit circle  to the real line. Consider a diagonal matrix
function with partial indices \{\(1\), \(-1\)\}
\begin{equation}
\label{eq:unstab}
\left(
 \begin{array}{cc}
  
\frac{t-i}{t+i} &  0 \\
0  & \frac{t+i}{t-i} 
 \end{array} \right)=\mathbf{I}
 \left(
 \begin{array}{cc}
  
\frac{t-i}{t+i} &  0 \\
0  & \frac{t+i}{t-i} 
 \end{array} \right)\mathbf{I}.
\end{equation}
Perturbing the matrix we have
\begin{equation}
\label{eq:unstab_1}
\left(
 \begin{array}{cc}
  
\frac{t-i}{t+i} &  0 \\
\epsilon  & \frac{t+i}{t-i} 
 \end{array} \right)=
\left(
 \begin{array}{cc}
  
1 &  \frac{t-i}{t+i} \\
0  & \epsilon 
 \end{array} \right)
 \mathbf{I}
\left(
 \begin{array}{cc}
  
0 &  -1/\epsilon \\
1  & \frac{t+i}{\epsilon(t-i)} 
 \end{array} \right).
\end{equation}
This example demonstrates that a small perturbations can not only
change the factors by an arbitrary amount but can also change the
partial indices (from \(\{1\), \(-1\}\) to \(\{0\), \(0\}\)). This is
significant because the partial indices are uniquely defined.  Note
that the sum of the partial indices remains the same. This is true in
general, which can be demonstrated if we equate the determinants of
both sides to reduced the problem to scalar factorisation. The partial
indices add to give the index~\eqref{eq:eq:index-func-defn} of the
determinant. In this case, the index of a function \(f\) is the
winding number of the curve \((\text{Re } f(t), \text{Im } f(t)) \),
\(t \in \mathbb{R}\). Hence, \(\mathrm{ind}(f)\) and, thus the sum of
partial indices, are stable under small perturbations.

\begin{rem}
  It is possible to use the non-uniqueness of
  factorisation~\cite{Spit} to obtain a different factorisation
  of~(\ref{eq:unstab})
\begin{equation}
\label{eq:unstab_2}
\left(
 \begin{array}{cc}
  \frac{t-i}{t+i} &  0 \\
0  & \frac{t+i}{t-i} 
 \end{array} \right)=
\left(
 \begin{array}{cc}
  1 &  \frac{t-i}{t+i} \\
0  & \epsilon 
 \end{array} \right) 
\left(
 \begin{array}{cc}
\frac{t-i}{t+i} &  0 \\
0  & \frac{t+i}{t-i} 
 \end{array} \right)
\left(
 \begin{array}{cc}
  1 & - \frac{t+i}{\epsilon(t-i)} \\
0  & 1/\epsilon
 \end{array} \right).
\end{equation}
This is more similar to~(\ref{eq:unstab}).
\end{rem}

The following surprising theorem provides the necessary and sufficient
conditions for the partial indecies to be invariant under sufficiently
small perturbations. 

\begin{thm}[Gohberg -- Krein, \cite{Spit_book}*{\S~6.2}]
The system \(\kappa_1 \ge \dots \ge \kappa_n\) of partial indices is stable if and only if
\[\kappa_1-\kappa_n\le 1.\]
\end{thm}

In fact, this condition is also sufficient for the stability of factors
in the Wiener norm.

\begin{thm}[Shubin, \cite{Spit_book}*{\S~6.6}]
  Assume the matrix function \(\mathbf{G}\) has a Wiener--Hopf
  factorisation and the tuple of its partial indices is
  stable. Then, for every \(\epsilon>0\) there exists a \(\delta>0\)
  such that, for \(\|\mathbf{F}-\mathbf{G}\|<\delta\), the matrix
  function \(\mathbf{F}\) admits a factorisation in which
  \(\|\mathbf{F}_{\pm}-\mathbf{G}_{\pm}\|<\epsilon\).
\end{thm}
An obstacle in using this result in applications is that one cannot in
general determine the partial indices without constructing the
factorisation.
The next section presents new conditions for stability of
factorisation for Daniele--Khrapkov matrices.












\section{Error estimates in Daniele--Khrapkov Matrices}
\label{sec:error}

This section examines function-theoretic factorisation of matrices of
Daniele--Khrap\-kov class~\eqref{eq:daniele-kharpkov-defn}. This class
was first considered by Khrapkov in connection to static stress fields
induced by notches in elastic wedges~\cite{Khrap}. There are other
numerous applications, e.g. related to wave
propagation~\cite{Application_DK,daniele2014wiener,Elastic_half}. 


Due to this special form~\eqref{eq:daniele-kharpkov-defn}
\(\mathbf{K}(t)\) can be re-expressed as
\[ \mathbf{K}(t)= r(t)  \big( \cosh[ \Delta (t)
\theta(t)] \mathbf{I} +\frac{1}{\Delta(t)} \sinh [ \Delta(t) \theta(t)]
\mathbf{J}(t) \big),\]
where:
\begin{equation}
  \label{eq:r-theta-defn}
  r(t)=\sqrt{1-\Delta^{2}(t)f^2(t)}, \qquad \theta(t)=\frac{1}{\Delta(t)} \ln\left( \frac{1+\Delta(t)f(t)}{1-\Delta(t)f(t)}\right).
\end{equation}
Multiplication of the above matrices is commutative, moreover
\[ \mathbf{K_1}(t) \mathbf{K_2}(t)= R(t)  \big( \cosh[ \Delta (t)
\Theta(t)] \mathbf{I}
+\frac{1}{\Delta(t)} \sinh [ \Delta(t) \Theta(t)]
\mathbf{J}(t) \big),\]
where:
\[R(t)=r_1(t)r_2(t), \quad \Theta(t)=\theta_1(t)+\theta_2(t).\]
This property is enough to obtain function-theoretical factorisation
\begin{equation}
\label{eq:D-K}
 \mathbf{K}_{\pm}(t)= r_{\pm} (t)  \big( \cosh[ \Delta (t)
\theta_{\pm}(t)] \mathbf{I} +\frac{1}{\Delta(t)} \sinh [ \Delta(t) \theta_{\pm}(t)]
\mathbf{J}(t) \big),
\end{equation}
where
\[r(t)=r_-(t)r_+(t), \quad \theta=\theta_-(t)+\theta_+(t).\] The
limitation is the degree of the polynomial \(\Delta^2\): if it is
greater than two then \( \cosh[ \Delta (t) \theta_{\pm}(t)]\) and
\(\sinh [ \Delta(t) \theta_{\pm}(t)]\) have exponential growth at
infinity~\cite{Ab_noncom}. This is an obstacle to the use of the
Wiener--Hopf technique.


We consider the question of stable factorisation for Daniele--Khrapkov
matrices in the following sense. Let \(K(t)\) and \(\tilde{K}(t)\) be
of Daniele--Khrapkov type and suppose \(\|K(t)-\tilde{K}(t)\|_2\) is
small. We provide an estimate on
\(\|K_{\pm}(t)-\tilde{K}_{\pm}(t)\|_2\). This splits into three
parts. The first part is to establish estimates for \(\|r(t)
-\tilde{r}(t)\|_2\) and \(\|\theta(t)-\tilde{\theta}(t)\|_2\) defined
by~\eqref{eq:r-theta-defn}. The second is to apply the error estimates
to parameters \(r_\pm(t)\) and \(\theta_\pm(t)\) of the
factors. Lastly \(\|K_{\pm}(t)-\tilde{K}_{\pm}(t)\|_2\) can be
examined.

Consider the matrix function \(\mathbf{K}(t)\) and its perturbation \(\tilde{\mathbf{K}}(t)\)
\begin{displaymath}
\mathbf{K}(t)= \mathbf{I}+f(t) \mathbf{J}(t),\qquad  
\tilde{\mathbf{K}}(t)= \mathbf{I}+\tilde{f}(t) \mathbf{J}(t),
\end{displaymath}
such that \(\|\Delta(t)f(t)-\Delta(t)\tilde{f}(t)\|_2<\epsilon\).  In this
setup the perturbation of \(r(t)\) can be estimated as follows.

\begin{lem}
\label{lem:mul}
Let \(r=\sqrt{1-\Delta^{2}(t)f^2(t)}\) and
\(\tilde{r}=\sqrt{1-\Delta^{2}(t)\tilde{f}^2(t)}\). Suppose that the
winding number of \( (1-\Delta^{2}(t)f^2(t))\) is zero, then
for \(\|\Delta(t)f(t)-\Delta(t)\tilde{f}(t)\|_2<\epsilon\) the
following estimate holds
\[\|r-\tilde{r}\|_2   <\frac{ N}{m} \epsilon,\]
where \(m=\min_{\mathbb{R}}  \{|r(t)|, |\tilde{r}(t)|\}>0\) 
and \( N=\max_{\mathbb{R}}\{|\Delta(t)f(t)|, |\Delta(t)\tilde{f}(t)|\} 
<\infty \).
\end{lem}

\begin{rem}
The assumptions are natural since \(|r(t)|^2\)  is the determinant of
the matrix \(K\) which together with the determinant of its inverse is non-zero.  
\end{rem}

\begin{proof}
  Since winding number of \(
  (1-\Delta^{2}(t)f^2(t))\) is zero and \(\epsilon\) is small enough
  we have winding number of \( (1-\Delta^{2}(t)\tilde{f}^2(t))\) is
  also zero.  The square root for \(r(t)\) in~\eqref{eq:r-theta-defn}
  can be taken single valued.  In the inequality
  \begin{displaymath}
    |\sqrt{a}-\sqrt{b}|= \frac {|a-b|}{\sqrt{a}+\sqrt{b}}
    \leq   \frac {|a-b|}{2\min(\sqrt{a},\sqrt{b})}.
  \end{displaymath}
  we substitute \(a=1-\Delta^{2}(t)f^2(t)\) and
  \(b=1-\Delta^{2}(t)\tilde{f}^2(t)\). We also replace
  \(\min(\sqrt{a},\sqrt{b})\) by a smaller value \(m=\min_{\mathbb{R}}
  \{|r(t)|, |\tilde{r}(t)|\}>0\).  Integrating squares of the both
  sides over the real line we obtain
  \begin{eqnarray*}
    \|r-\tilde{r}\|_2&\le&
    \frac{1}{2m} \|\Delta^{2}(f^2-\tilde{f}^2)\|_2\\
    &\le&    \frac{1}{2m}
    \|(\Delta(f+\tilde{f}))(\Delta(f-\tilde{f}))\|_2\\
    &\le&     \frac{1}{2m}
    \left(\int_{\mathbb{R}} |\Delta(t)(f(t)+\tilde{f}(t))\Delta(t)(f(t)-\tilde{f}(t))|^2\,dt\right)^{1/2}\\
    &\le&    \frac{2N}{2m}
    \left(\int_{\mathbb{R}} |\Delta(t)(f(t)-\tilde{f}(t))|^2\,dt\right)^{1/2}\\
    &\le&\frac{N}{m} \epsilon,
  \end{eqnarray*}
  since \(|\Delta(t)f(t)|\) and \(|\Delta(t)\tilde{f}(t)|\) are bounded by  \(N=\max_{\mathbb{R}}\{|\Delta(t)f(t)|, |\Delta(t)\tilde{f}(t)|\} \).
\end{proof}

Similarly the behaviour of \(\theta\) under perturbation is important.

\begin{lem}
\label{lem:add}
Let
 \[\theta(t)=\frac{1}{\Delta(t)} \ln\left(
 \frac{1-\Delta(t)f(t)}{1+\Delta(t)f(t)}\right),\quad
 \tilde{\theta}(t)=\frac{1}{\Delta(t)} \ln\left(
 \frac{1-\Delta(t)\tilde{f}(t)}{1+\Delta(t)\tilde{f}(t)}\right). \]
 Suppose that the winding number of
 \(\big(\frac{1-\Delta(t)f(t)}{1+\Delta(t)f(t)} \big)\) is zero, then
 for small \(\|\Delta(t)f(t)-\Delta(t)\tilde{f}(t)\|_2<\epsilon\) the
 following estimate holds
 \[\|\theta -\tilde{\theta}\|_2 <\frac{2 \epsilon}{c d^2 L}, \] where
 \(d=\min_{\mathbb{R}}\{|1+\Delta(t)f(t)|,
 |1+\Delta(t)\tilde{f}(t)|\}>0\) and \(L=\max_{\mathbb{R}}
   \left|\frac{1-\Delta(t)f(t)}{1+\Delta(t)f(t)}\right|\) ,
   \(c=\min_{\mathbb{R}}|\Delta(t)| >0\).
\end{lem}
 
 \begin{rem}
   Since \(\Delta\) has no zeroes on the real line we can assume
   \(\text{min}|\Delta|\ge c>0\).  Also note that
   \(|1+\Delta(t)f(t)|\) and \(|1-\Delta(t)f(t)|\) are non-zero and
   finite respectively since they are multiples of det\(K\).
 \end{rem}
 
\begin{proof}
  From the assumption on zero winding number, the logarithms in the
  definition \(\theta(t)\) and \(\tilde{\theta}(t)\) are single valued
  functions.  The the mean value theorem applied to the logarithm
  function provides an inequality:
  \begin{displaymath}
    |\ln a- \ln b| \leq \frac{|a-b|}{\min(a,b)}.
  \end{displaymath}
  We substitute \(\ln a=\Delta (t)\theta(t)\), \(\ln b= \Delta
  (t)\tilde{\theta}(t) \) and replace \(\min(a,b)\) by \(L\) defined
  in the statement. Then, squaring both sides and integrating over the
  real line we obtain:
  \begin{eqnarray*}
    \|\theta -\tilde{\theta}\|_2 &\le &\frac{1}{cL}\left\|
    \frac{1-\Delta(t)f(t)}{1+\Delta(t)f(t)}-
    \frac{1-\Delta(t)\tilde{f}(t)}{1+\Delta(t)\tilde{f}(t)}\right\|_2\\
    &\le &\frac{2}{cL}\left\|\frac{\Delta(t)f(t)-\Delta(t)\tilde{f}(t)}{(1+\Delta(t)f(t))(1+\Delta(t)\tilde{f}(t))}\right\|_2\\
    &\le & \frac{2}{c d^2 L}\|\Delta(t)f(t)-\Delta(t)\tilde{f}(t)\|_2,
  \end{eqnarray*}
  where \(c\) and \(d\) are defined in the statement.
\end{proof}

Now we are in the position to apply the scalar error estimates.
Under the assumptions of the above Lemma~\ref{lem:add} and using the
additive error estimates Theorem~\ref{thm:add} we obtain:
\begin{equation}
  \label{eq:theta-estimation}
  \|\theta_{\pm} -\tilde{\theta}_{\pm}\|_2 <\frac{2}{c d^2 L} \epsilon. 
\end{equation}
Using Lemma~\ref{lem:mul}  and  the multiplicative error estimates
Theorem~\ref{thm:mul} it follows that
\begin{equation}
  \label{eq:r-estimation}
  \|r_{\pm} -\tilde{r}_{\pm} \|_2   <\frac{5 M N}{ m^{2}} \epsilon,
\end{equation}
where \(M=\max_{\mathbb{R}}  \{|r(t)|, |\tilde{r}(t)|\}>0\).

To simplify calculation in the next theorem, we will asssume that
\[\mathbf{J}=\left(
 \begin{array}{cc}
0 &  k_1 \\
 k_2 & 0 
\end{array} \right),\]
is a constant matrix.
Then, a sufficiently small
\(\|\Delta(t)f(t)-\Delta(t)\tilde{f}(t)\|_2\) guarantees that 
\(\|\mathbf{K}-\tilde{\mathbf{K}}\|_2\) is small as well.


 \begin{thm}
   Let \(\mathbf{K}\) and \(\tilde{\mathbf{K}}\) be of the above form,
   \(\|\Delta(t)f(t)-\Delta(t)\tilde{f}(t)\|_2<\epsilon\) and
   \(\Delta(t)=C\), satisfying the assumptions of Lemmas~\ref{lem:mul}
   and~\ref{lem:add} .  Then, the error
   \(\|\mathbf{K}_{\pm}-\tilde{\mathbf{K}}_{\pm}\|_2\) is a linear
   function of \(\epsilon\) and the exact estimates can be obtained
   using the above scalar estimates.
 \end{thm}

\begin{proof}
  Let \(a_{11}\) and \(\tilde{a}_{11}\) are the top-left elements of
  \(\mathbf{K}\) and \(\tilde{\mathbf{K}}\) respectively.  Then
\begin{eqnarray*}
\|a_{11}-\tilde{a}_{11}\|_2 &= &\| r_{\pm} (t)\cosh[ \Delta (t)
\theta_{\pm}(t)]-  \tilde{r}_{\pm} (t)\cosh[ \Delta (t)
\tilde{\theta}_{\pm}(t)] \|_2\\ &\le & \|r_{\pm} (\cosh[ \Delta (t)
\theta_{\pm}(t)]-\cosh[ \Delta (t)
\tilde{\theta}_{\pm}(t)])\|_2\\
&&+{}\|\cosh[ \Delta (t)
\tilde{\theta}_{\pm}(t)](r_{\pm}-\tilde{r}_{\pm} )\|_2,
\end{eqnarray*}
where the triangle inequality was used. Then, using the mean value
theorem for \(\cosh\) we obtain
\begin{eqnarray*}
\|a_{11}-\tilde{a}_{11}\|_2  &\le &|r_{\pm}| \text{ }|\sinh[ \Delta (t)
\theta_{\pm}(t)]|\text{ } \|\Delta (t)
\theta_{\pm}(t)- \Delta (t)
\tilde{\theta}_{\pm}(t)\|_2\\
&&+{}|\cosh[ \Delta (t)
\tilde{\theta}_{\pm}(t)]|\text{ } \|(r_{\pm}-\tilde{r}_{\pm} )\|_2.
\end{eqnarray*}
To complete the calculation it is enough to use the bound for \(|r_{\pm}|\),
\(|\sinh[ \Delta (t) \theta_{\pm}(t)]|\) and \(|\cosh[ \Delta
(t)\tilde{\theta}_{\pm}(t)]|\). This follows from \(r_{\pm}\) and
\(\theta_{\pm}\), being bounded, having zero winding number and
tending to a constant~\cite{My1}. The calculations for other entries
\(\|a_{ij}-\tilde{a}_{ij}\|_2\), \(i,j=1,2\) are performed analogously. All the
norms of \(2\times 2\) matrices are equivalent so it does not matter
which one is chosen.
\end{proof}

In the subsequent Sections we present several situations where our
results may be applied.  Numerical example will be presented in
Section~\ref{sec:numerical}.

\section{Approximate reducing to extended Daniele--Khrapkov}
\label{sec:reduce}



\subsection{Exact reduction to Daniele--Khrapkov Matrices}

The most general class of matrix functions which can be factored
using the above technique is
\begin{equation}
\label{eq:Reduce_1}
\mathbf{K} = \mathbf{S_{+}} \big(g_1 \mathbf{I} +g_2 \mathbf{J} \big)\mathbf{S_{-}}. 
\end{equation}
with \(\mathbf{S_{+}}\) and \(\mathbf{S_{-}}\) analytic in the
upper and lower half-plane respectively and \(\mathrm{tr}\mathbf{J}=0\).

This can be rearranged as
\begin{equation}
\label{eq:Reduce_2}
\mathbf{K} = g_1 \mathbf{S_1} +g_2 \mathbf{S_2},\qquad
\text{ where }\quad
\mathbf{S_1}=\mathbf{S_{+}}\mathbf{S_{-}} \text{ and }
 \mathbf{S_2}=\mathbf{S_{+}} \mathbf{J} \mathbf{S_{-}} . 
\end{equation}
The challenge is to work backwards from Equation~\eqref{eq:Reduce_2}
to~\eqref{eq:Reduce_1}.  The first step is the factorisation of
\(\mathbf{S_1}=\mathbf{S_{+}}\mathbf{S_{-}}\) and second step is to
ensure the second term satisfies the necessary conditions for
\(\mathbf{J}=\mathbf{S}_{+}^{-1} \mathbf{S}_2 \mathbf{S}_{-}^{-1}\).
To satisfy these considerations one can take \(\mathbf{S_1}\) and
\(\mathbf{S_2}\) to be rational, this class was studied
in~\cite{Speck_ind}.

Now we will outline the procedure to reduce Equation~\ref{eq:Reduce_2}
to~\ref{eq:Reduce_1}. Initially one must rule out the case when
\(\mathbf{S_1}\) has a zero on the real line. Since the matrix
\(\mathbf{K}\) does not have any zeros, any zeros of \(\mathbf{S}_1\)
must be compensated either by multiplying by \(f_1\) or by adding
\(f_2 \mathbf{S_2} \). So by constructing a different linear
combination it can be assumed that \(\mathbf{S_1}\) is non-zero on the
real line. Then using the rational factorisation
\(\mathbf{S_1}=\mathbf{S_{+}}\mathbf{S_{-}}\) we obtain
\[\mathbf{K} =\mathbf{S_{+}} \big(g_1 \mathbf{I}+g_2 \mathbf{R}\big)\mathbf{S_{-}} ,\]
with \(\mathbf{R}=\mathbf{S_{+}}^{-1} \mathbf{S_2} \mathbf{S_{-}} ^{-1}\).

This can it can be re-written as
\[\mathbf{K} =\mathbf{S_{1+}} \big(f_1 \mathbf{I}+f_2 \mathbf{J} \big)\mathbf{S_{1-}} ,\]
where \(\mathbf{J} =\mathbf{R}-1/2\text{tr}(\mathbf{R})\) for some new
functions \(f_1\) and \(f_2\), see~\cite{Speck_ind} for further
details. We will call such matrices \emph{extended} Daniele--Khrapkov class.

\subsection{Approximate reduction to Daniele--Khrapkov}

We give a description of a larger class of matrices which may
approximately factorised through approximation by matrix functions
from the extended Daniele--Khrapkov class~(\ref{eq:Reduce_2}). Those
matrices have the property that every entry of the matrix has
elements of the form:
\[f_1r_{ij}^1+f_2r_{ij}^2,\] with two fixed arbitrary functions
\(f_1\) and \(f_2\) and rational functions \(r_{ij}^1\) and
\(r_{ij}^2\). In the whole generality it shall be discussed
elsewhere. Here, we concentrate on a subclass, related to
work~\cite{Ab_noncom} with interesting
applications~\cite{Elastic_half}. This subclass allows to overcome the
problem of exponential growth of the factors in the Daniele--Khrapkov
matrices for high degree of polynomial \(\Delta(t)\). This approximate
procedure is simpler than the exact one provided by
Daniele~\cite{daniele2014wiener}*{\S~4.8.5}.

Let us begin with matrix
\[\mathbf{K}(t)= \mathbf{I}+f(t)  
 \left(
 \begin{array}{cc}
  
0 &  n(t) \\
 p(t) & 0 
 \end{array} \right).
\]
We can rearrange it into the form
\[\mathbf{K}(t)= \mathbf{I}+g(t)  \mathbf{J}(t), \]
 with:
\[\mathbf{J}(t)= \left(
 \begin{array}{cc}
  
0 &  \big(\frac{n(t)}{p(t)}\big)^{1/2} \\
\big(\frac{p(t)}{n(t)}\big)^{1/2}  & 0 
 \end{array} \right),
\]
and \(g(t)=f(t)\big(\frac{n(t)}{p(t)}\big)^{1/2}\).
The advantage of this rearrangement being,
\[\mathbf{J}^2(t)=\mathbf{I},\]
and the disadvantage is that now \(\mathbf{J}\) has branch cut
singularities. To overcome that Abrahams proposed to rationally
approximate \(\big(\frac{p(t)}{n(t)}\big)^{1/2}\) by \(r_N(t)\)
giving
\[\mathbf{J}_N(t)= \left(\begin{array}{cc}
  0 &  1/r_N(t) \\
r_N(t)  & 0 
 \end{array} \right).
\]
This procedure is exact when \(n(t)\) and \(p(t)\) have perfect
squares as factors.

The approximate matrix can be decomposed as in~(\ref{eq:D-K})
\[\mathbf{K}_N(t)=\mathbf{I}+g(t)
\mathbf{J}_N(t)=\mathbf{Q}_{N-}\mathbf{Q}_{N+}, \]
but the factors \(\mathbf{Q}_{N\pm}\) have poles. Hence, a meromorphic
factorisation is obtained. 
\begin{rem}
  Error bounds~\eqref{eq:theta-estimation} and~\eqref{eq:r-estimation}
  on \(\theta_{\pm}\) and \(r_{\pm}\) still hold in this meromorphic
  factorisation.
\end{rem}
To remove poles we can consider the factorisation
\begin{equation}
  \label{eq:M-defn}
  \mathbf{K}_N(t)=(\mathbf{Q}_{N-}\mathbf{M})(\mathbf{M^{-1}}\mathbf{Q}_{N+}),
\end{equation}
where \(\mathbf{M}\) is a rational matrix, which is chosen such that
the resulting factorisation has no poles in the required half-planes,
see~\cite{Ab_noncom} for further details. We are turning to illustrations
of this method.

\begin{ex}
  This example is concerned with the earlier example of
  instability~(\ref{eq:unstab}). The aim is to show that although the
  indices are \(1\) and \(-1\) it is still possible to have a stable
  perturbation. The construction is based on the results from the
  previous sections.
\begin{eqnarray*}
\label{eq:stab}
\left(\begin{array}{cc}
  \frac{t-i}{t+i} &  \epsilon f(t) \\
c \epsilon f(t) & \frac{t+i}{t-i} 
 \end{array} \right)&=&
 \left(\begin{array}{cc}
  \frac{t-i}{t+i} & 0  \\
 0 & \frac{t+i}{t-i} 
 \end{array} \right)+ \epsilon f(t) 
 \left(\begin{array}{cc}
 0 &  1 \\
c   & 0 
 \end{array} \right),\\
&=& 
\left(\begin{array}{cc}
  t+i &  0 \\
0  & \frac{1}{t+i} 
 \end{array} \right)
\mathbf{K}
\left(\begin{array}{cc}
  \frac{1}{t-i} &  0 \\
0  & t-i
 \end{array} \right).
\end{eqnarray*}
with
\[\mathbf{K}= \mathbf{I} + \epsilon f(t) 
\left(\begin{array}{cc}
  0 & (t-i)^{-1}(t+i)^{-1}  \\
c (t-i)(t+i)  & 0
\end{array} \right). \]
The matrix \(\mathbf{K}\) is of Abrahams type with the ratio of the off-diagonal
elements being a
square. Hence, there is no need for rational approximation and the 
procedure is exact in this case. One can construct the factors using the
(\ref{eq:D-K}). Lemmas~\ref{lem:mul} and~\ref{lem:add} can be applied
when \(f\) satisfies their assumptions. Hence, a meromorphic factorisation has been obtained which is
stable for small \(\epsilon\).
Then, the final step is to construct a matrix \(\mathbf{M}\) as in~\eqref{eq:M-defn}. In the case
when  \(f(t)\equiv k\)  the matrix \(\mathbf{M}\) takes the form
\begin{equation}
\label{eq:unstab_3}
\mathbf{M}=
\left(
 \begin{array}{cc}  
1+\frac{i/2}{t-i}+\frac{i/2k}{t+i} & \frac{i/2}{t-i}+\frac{i/2k}{t+i}  \\
1  & 1 
 \end{array} \right),
 \end{equation}
with det \(\mathbf{M}=1\). This completes the factorisation of the
perturbed matrix.
\end{ex}

\section{Numerical Results}
\label{sec:numerical}
This section presents two approximate scalar factorisations with
different indices and these are used to construct two approximate
Daniele--Khrapkov factorisations. 

\subsection{Rational approximation}

Rational approximation of functions has its uses in Wiener--Hopf
factorisation. One example was mentioned in previous section.
Paper~\cite{My1} applyies rational approximation to simplify the
scalar factorisation and avoid calculations of a Cauchy type
integral. 




 

Rational approximation is useful for Daniele--Khrapkov factorisation
because once the approximations for \(K_1\) and \(K_2\) are obtained
algebraic expressions such as
\begin{displaymath}
  K_1+c,\qquad K_1+K_2,\qquad K_1K_2,
\end{displaymath}
can be factored easily. This is not true in general as can be
seen from the next two examples.
\begin{figure}[htbp]

 \includegraphics[scale=0.8,angle=0]{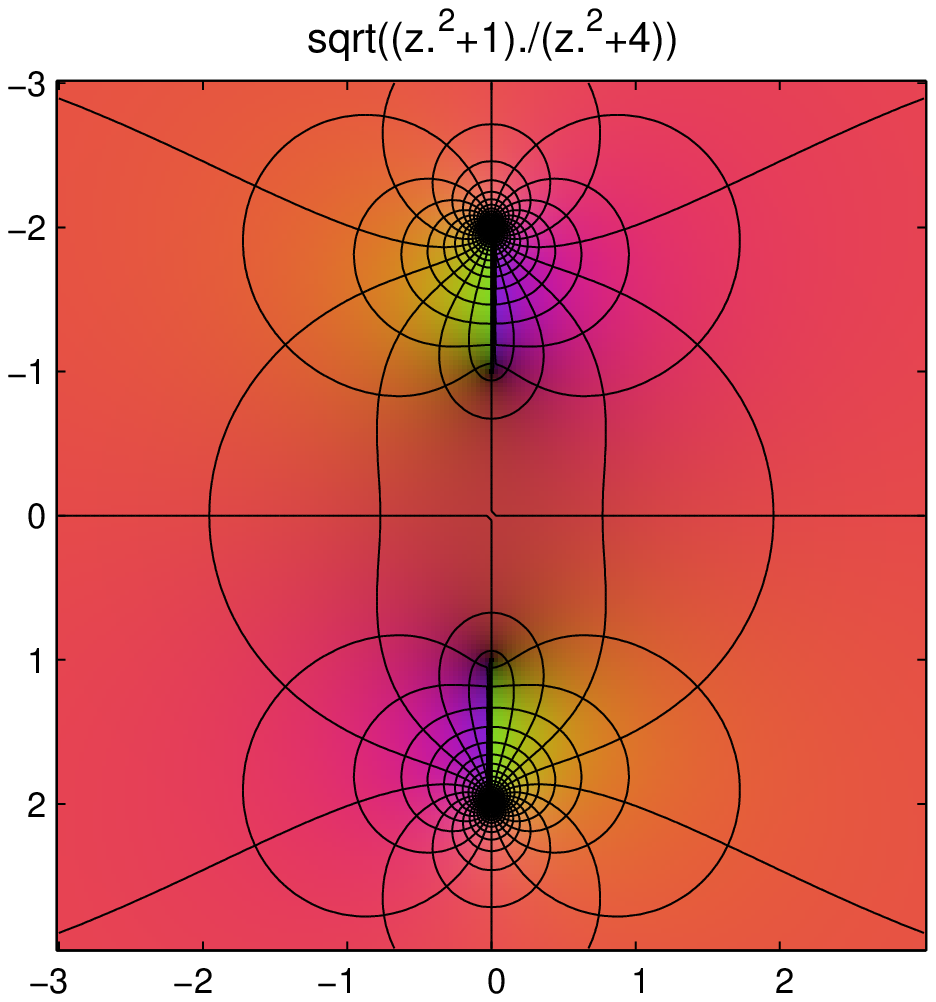} 
\includegraphics[scale=0.8,angle=0]{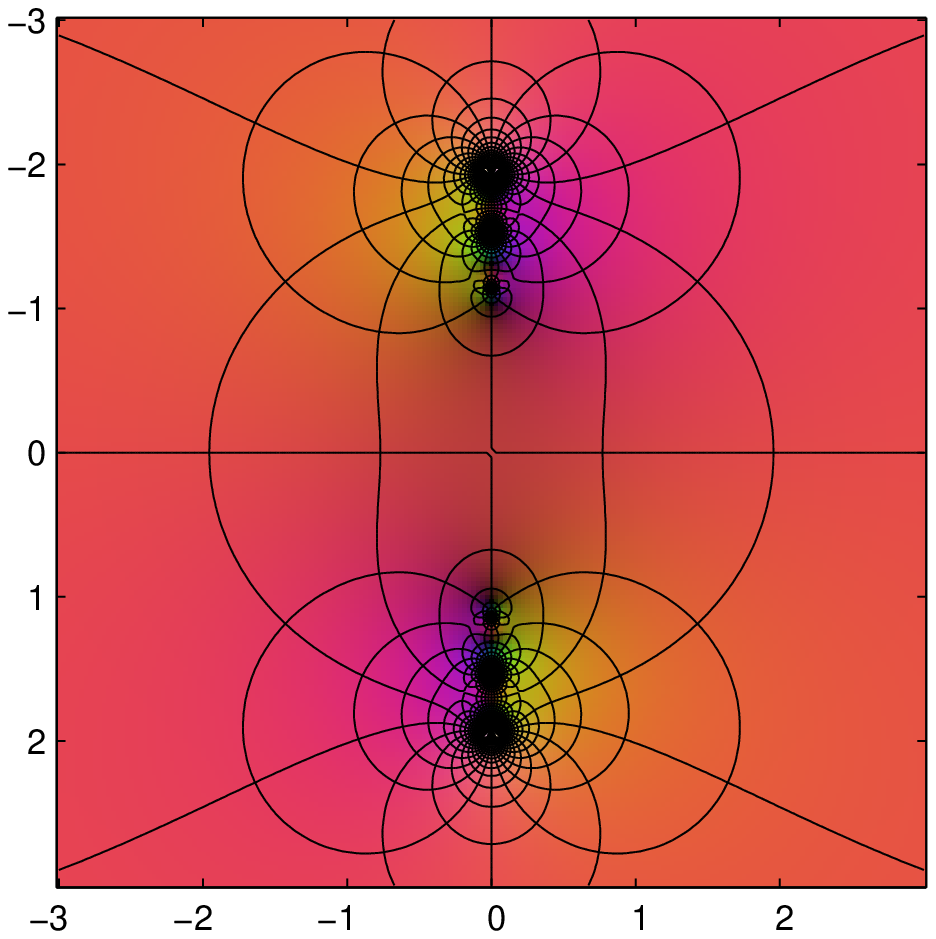}
\caption{Contour lines for the real and imaginary parts of function
  $F$~\eqref{eq:F-second-ex} and its rational approximation
  \(\tilde{F}\). They  are superimposed on a full colour image
  using a colour scheme developed by John Richardson. {\color{red}
    Red} is real, {\color{blue} blue} is positive imaginary,
  {\color{green} green} is negative imaginary, black is small
  magnitude and white is large magnitude.  Branch cuts appear as
  colour discontinuities and coalescent contour lines. Produced using MATLAB package \texttt{zviz.m}. }
 \label{fig:vis} 
\end{figure}
\begin{ex}
  Consider the function with zero index
  \begin{equation}
    \label{eq:F-second-ex}
   F(t)=\sqrt{    \frac{t^2+1}{t^2+k^2}}, 
  \end{equation}
 and with finite branch cuts from \(i\)
  to \(ki\) and from \(-i\) to \(-ki\).  This function is closely
  associated with the matrix function factorisation from problems in
  acoustics and elasticity, see~\cite{Pade}.  The factors can easily
  be seen by inspection
  \[F_{\pm}(t)=\sqrt{ \frac{(t\pm i)}{(t\pm ik)}}, \qquad
  F_{+}(t)=F_-(-t).\] However, the factorisation of \(F(t)+1\) cannot
  be achieved by inspection.  Rational approximation of \(\sqrt{
    \frac{t^2+1}{t^2+4}}\) had been also extensively studied
  in~\cite{My1}.  The approximation was achieved by constructing an
  appropriate transformation from the whole real line to the unit
  interval. As a result, an approximate factorisation has a small
  global error (\(10^{-12}\) on the real line).  Here, we produce
  Figure~\ref{fig:vis}, which demonstrates the closedness of
  approximation on the whole complex plane.
\end{ex}

\begin{ex}
Let us consider rational approximation of the function 
\begin{equation}
  \label{eq:K-third-ex}
  K=\sqrt{     \frac{(t+2i)(t+3i)}{(t-2i)(t-3i)}}
\end{equation}
with the index \(-1\). Again, the function has been chosen to have the
explicit exact factorisation
\[\sqrt{ \frac{(t+2i)(t+ki)}{(t-2i)(t-ki)}}=
\frac{\sqrt{(t+2i)(t+ki)}}{t+i}\left(\frac{t-i}{t+i}\right)^{-1}
\frac{t-i}{\sqrt{(t-2i)(t-ki)}}.\] The function-theoretic
factorisation has growth at infinity, making it more difficult to
approximate. Nevertheless, it can be rationally approximated and the
error \(|K-\tilde{K}|\) is presented in Figure~\ref{fig:err}.
Importantly, the error of the factors \(|K_{\pm}-\tilde{K}_{\pm}|\) is
also small (Figure~\ref{fig:err_fac}).  For more details on rational approximation of complex
valued functions see~\cite{Tref_old}.
\end{ex}

\begin{figure}[htbp]
\centering
\includegraphics[scale=0.55,angle=0]{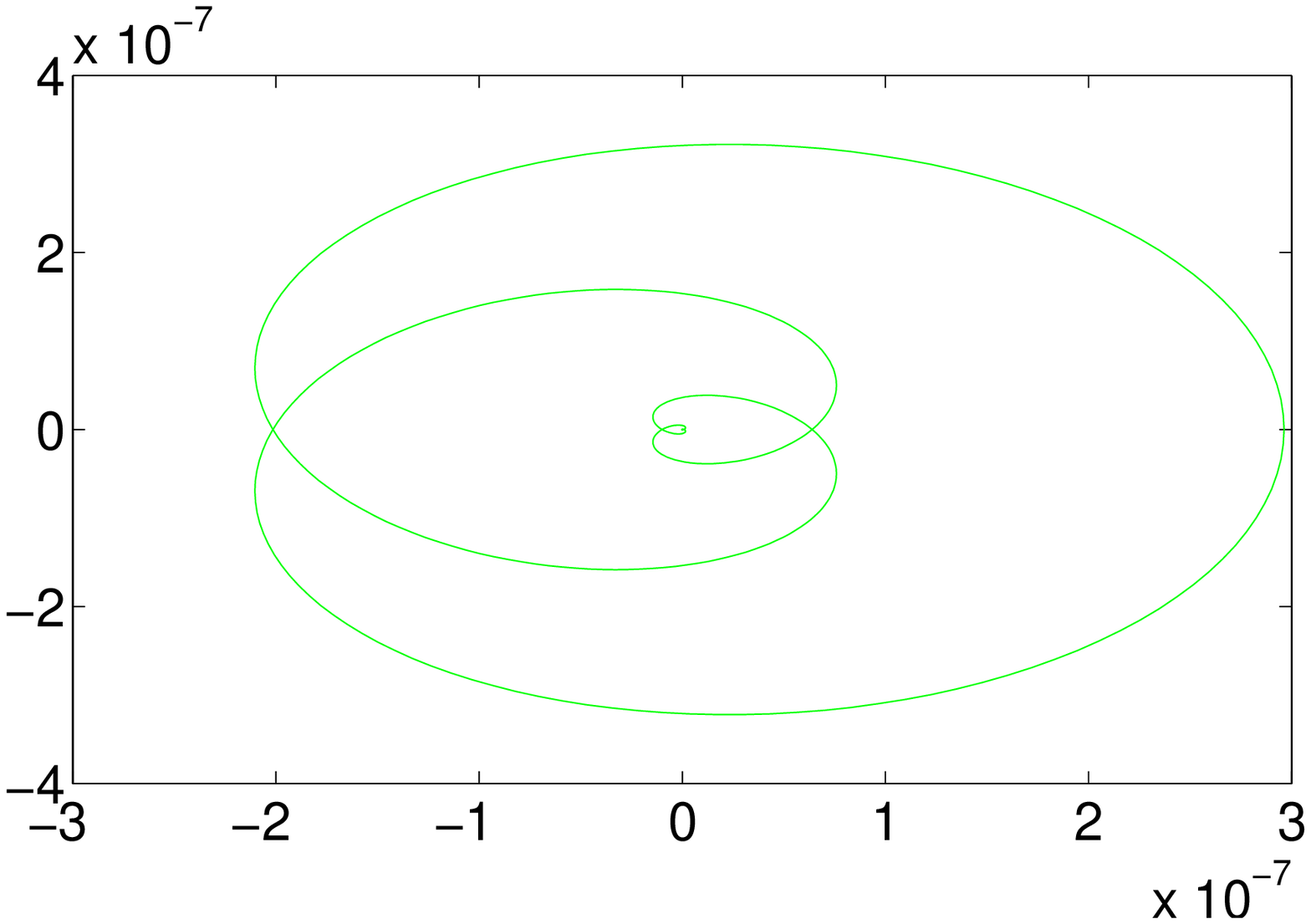}
\caption{Error in approximating function $K$~\eqref{eq:K-third-ex} by
  \([8,8]\) plotted as real against imaginary part. The accuracy of an
  approximation is denoted by the size of the disc the curve is
  contained in.}
\label{fig:err}
\end{figure}

\begin{figure}[htbp]
 \centering
  \includegraphics[scale=0.55,angle=0]{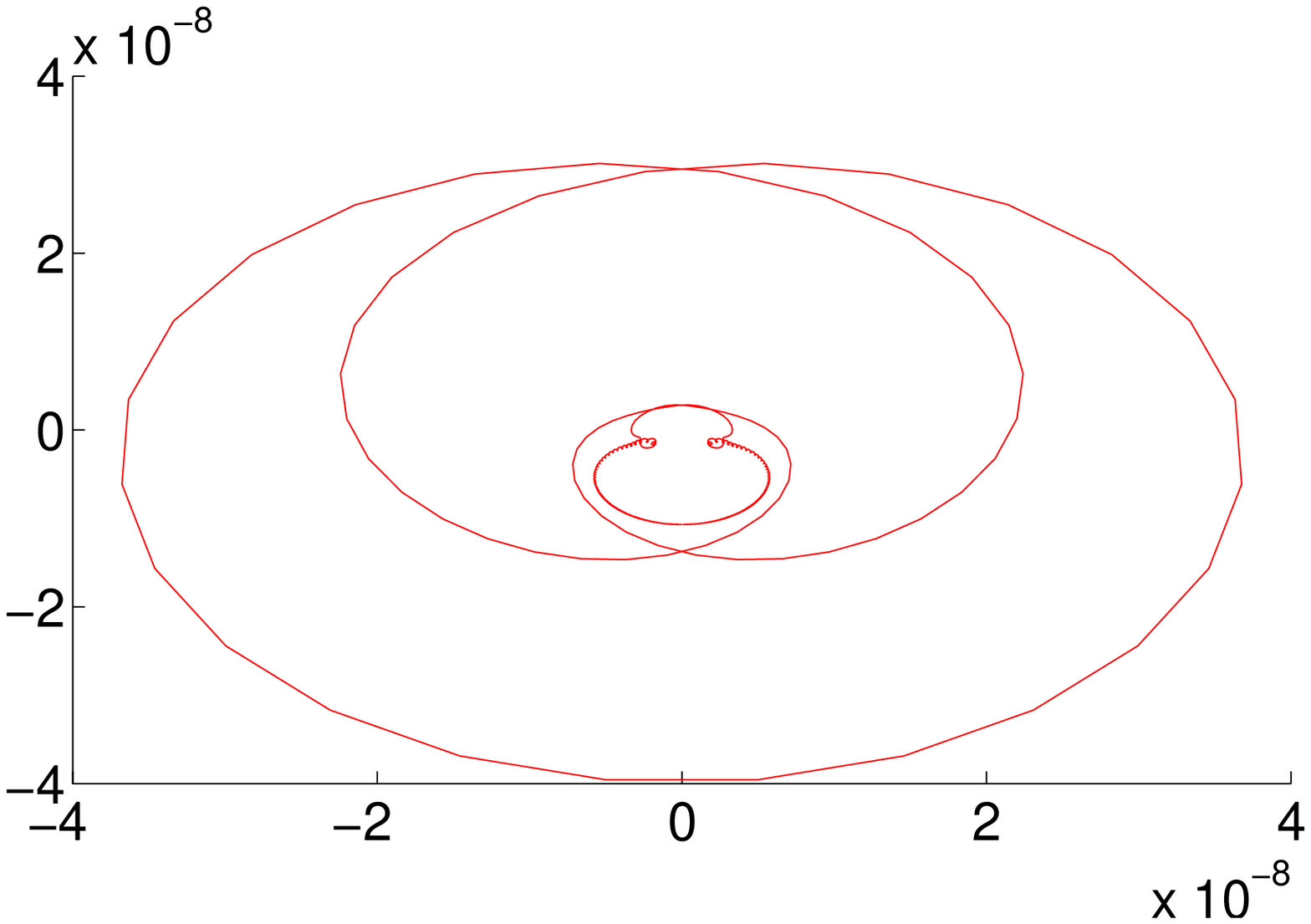}
  \caption{Error of factor \(K_\pm\) on the real line plotted as real against
    imaginary part. The accuracy of an approximation is denoted by the
    size of the disc the curve is contained in.}
  \label{fig:setup}
\label{fig:err_fac}
\end{figure}

\subsection{Numerical  Matrix Factorisation}
The stability result from Section~\ref{sec:error} can be used in
numerical computations. Two different examples are presented. For each
example the Daniele--Khrapkov factorisation is computed in two
different ways. The first method is the direct use of a Cauchy
integral to calculate the scalar factorisation of \(r_{\pm}\) and
splitting \(\theta_{\pm}\). So the initial matrix is exact, the
factors have errors due to computation of Cauchy integrals.  In the
second method the entries of the matrix are rationally approximated
and for this matrix the exact Daniele--Khrapkov factorisation is
obtained. The matrix is approximate but the factorisation of this
matrix is exact. The first method will be referred as ``exact'' and
the second one as ``approximate'' although the reader should note that both are approximate factorisations. The results of these
two methods are then compared for each example.

The first example is
\[\mathbf{K}_1(t)= 
\mathbf{I}+\sqrt{ \frac{t^2+1}{t^2+4}}   \left(\begin{array}{cc}  
0 &  1 \\
-2  & 0 
 \end{array} \right).\]
The ideas is to rationally approximate \(\sqrt{
  \frac{t^2+1}{t^2+4}}\) by \(f_N\). Then the factorisation of:
 \[
 \left(\begin{array}{cc}  
1 &  f_N \\
cf_N  & 1
\end{array} \right),\]
is computed and compared with the ``exact'' factorisation. The
advantage of such an approximation is 
that there is no need to use the Cauchy formula to find \(r_{\pm}\)
and \(\theta_{\pm}\). Note that the 
approximate matrix has all rational entries and hence in theory
factorisation can be achieved using methods for rational matrix
functions. But in practice the implemented procedures are unstable,
making it impossible. At present, very few implemented algorithms
Wiener--Hopf exist.  For
 example, there has been some attempts recently~\cite{Ann_rat} to
 produce numerical factorisation algorithms for rational matrix functions and
 numerical algorithms for Riemann-Hilbert problems~\cites{Olver_steep, Fokas_Painleve,MyStrip}.

The second example is
\[\mathbf{K}_2(t)= \mathbf{I}+\sqrt{ \frac{(t+2i)(t+i)}{(t-2i)(t-i)}} \left( \begin{array}{cc}  
0 &  1 \\
-2  & 0 
 \end{array} \right).\]
Similarly the approximate factorisation is considered by approximating
\(\sqrt{ \frac{(t+2i)(t+i)}{(t-2i)(t-i)}}\).

\begin{figure}[htbp]
  \centering
 \includegraphics[scale=0.31,angle=0]{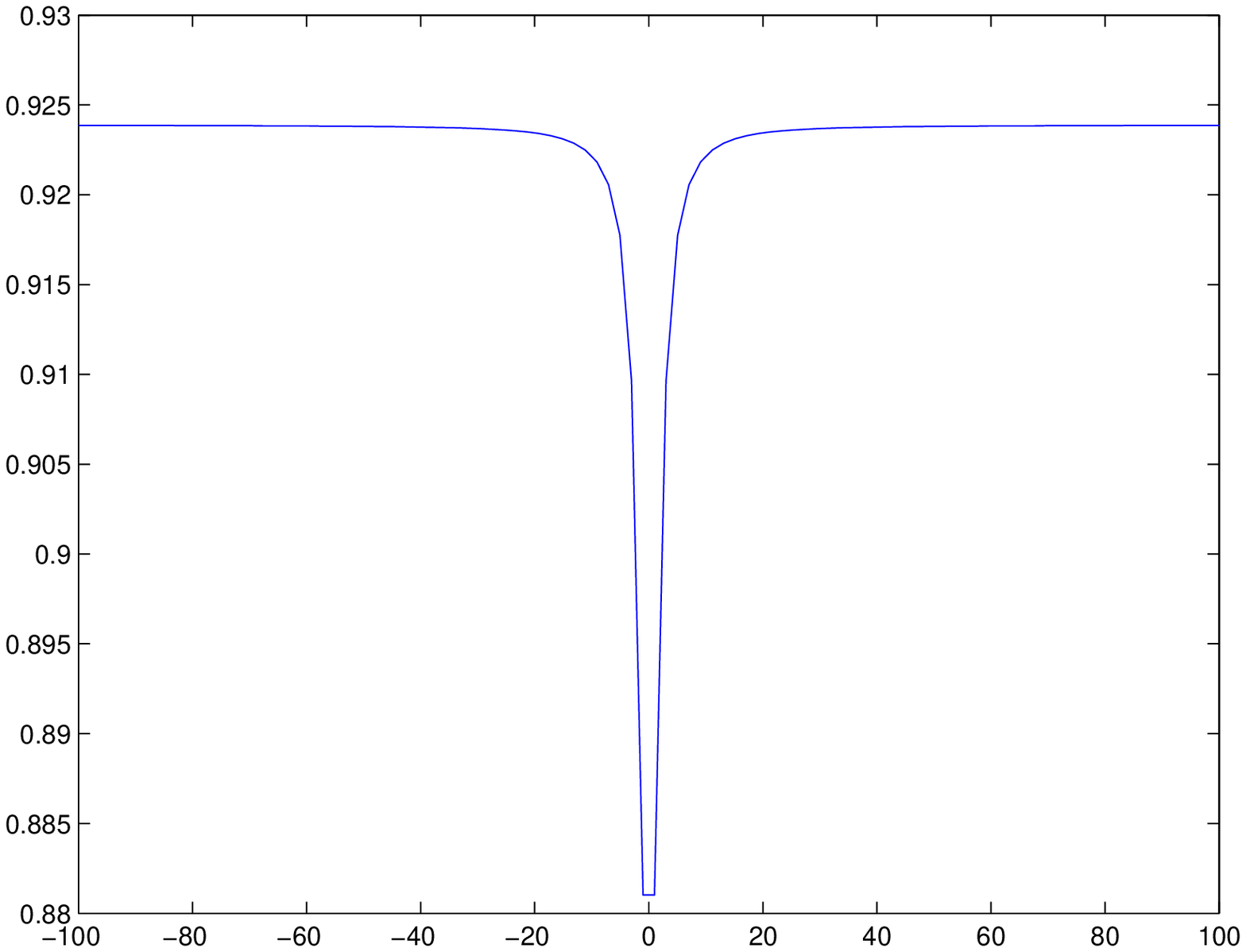}\hfill
 \includegraphics[scale=0.31,angle=0]{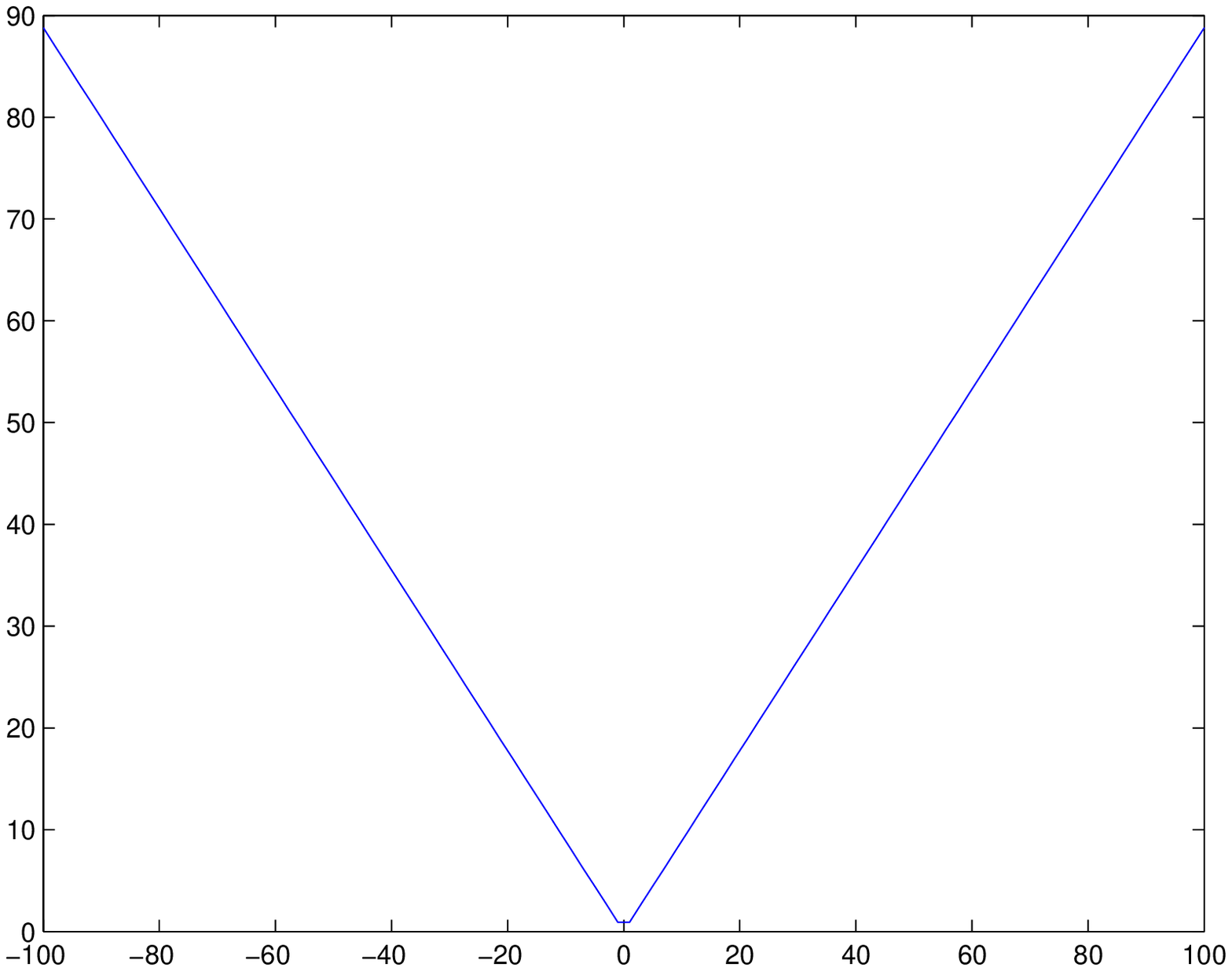}
  \caption{The modulus of \(K_{1+}\) and \(K_{2+}\) on the real line.}
\label{fig:beh}
  \end{figure}

  The difference in behaviour on the real line of the two examples can
  be seen in Figure~\ref{fig:beh}. This is because their
  partial indices are different.  The first example have partial indices
  \(\{0\), \(0\}\) and the second \(\{-1\), \(-1\}\). These partial
  indices can be computed using the following identity
 \[
 \left(\begin{array}{cc}  
1 &  f \\
cf  & 1
 \end{array} \right)=
 \left(\begin{array}{cc}  
1 &  1\\
c^{1/2}  & -c^{1/2}
 \end{array} \right)
  \left(\begin{array}{cc}  
1+c^{1/2}f &  0\\
0 & 1-c^{1/2}f
 \end{array} \right)
  \left(\begin{array}{cc}  
\frac{1}{2} &  \frac{c^{-1/2}}{2}\\
\frac{1}{2}  & \frac{-c^{-1/2}}{2}
 \end{array} \right).
 \]
 The errors are compared in Figure~\ref{fig:ex_1} and
 Figure~\ref{fig:ex_2}. It should be noted that the calculation of
 ``approximate'' factors took significantly less computational time
 than the ``exact'' factors.  Besides the natural difference in
 magnitude of errors (due to the difference in errors of rational
 approximations) the shape of the curves are dramatically
 different. It seems the error in Figure~\ref{fig:ex_1} is random and
 in Figure~\ref{fig:ex_2} is systemic. This suggests that in the first
 example the error in ``exact'' factorisation is greater than
 ``approximate'' factorisation. So the accumulated errors in computing
 Cauchy integrals is greater than the error in once approximating
 entries of the matrix function. The reverse is true in the second
 example.

\begin{figure}[htbp]
  \centering
 \includegraphics[scale=0.6,angle=0]{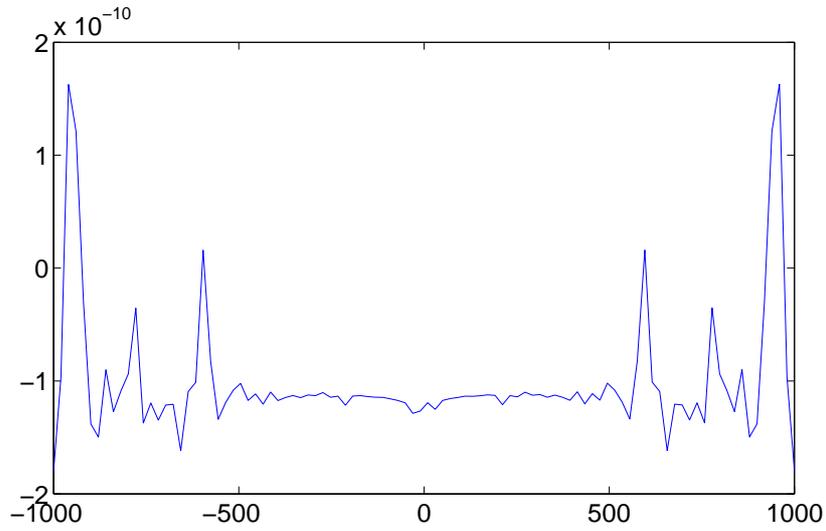}\hfill 
 \caption{The modulus of the difference in $a_{11}$ elements of ``exact'' and  ``approximate''  factors for $K_1$.}
 \label{fig:ex_1} 
\end{figure}

\begin{figure}[htbp] 
  \centering
 \includegraphics[scale=0.6,angle=0]{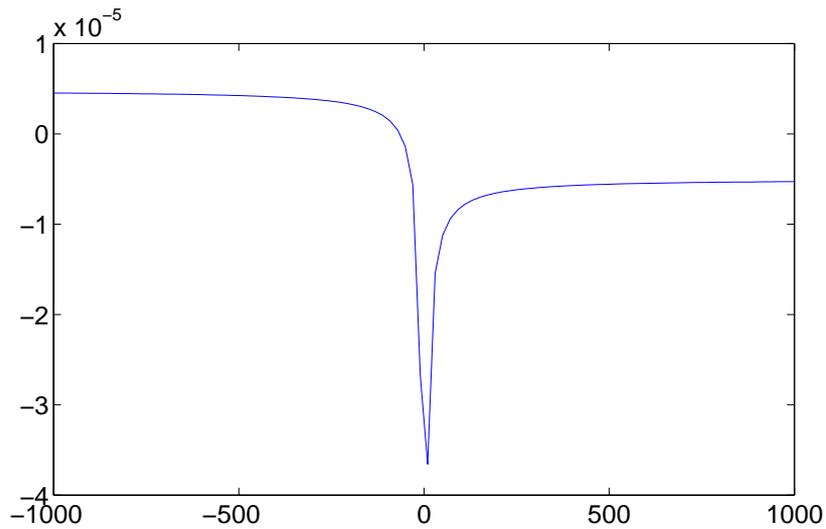}
 \caption{The modulus of the difference in $a_{11}$ elements of ``exact'' and  ``approximate'' factors for $K_2$.}
\label{fig:ex_2} 
\end{figure}




\section{Acknowledgements}

I am grateful for support from Prof. Nigel Peake. I benefited from
useful discussions with Dr Rogosin, Prof. Speck and Prof. Spitkovsky. Suggestions
 of the anonymous referees helped to improve this paper. This work was supported by the UK Engineering and Physical Sciences Research Council (EPSRC) grant EP/H023348/1 for the University of Cambridge Centre for Doctoral Training, the Cambridge Centre for Analysis.


\bibliography{newgeometry}

\begin{bibdiv}
\begin{biblist}

\bib{Elastic_half}{article}{
      author={Abrahams, I.~D.},
       title={Radiation and scattering of waves on an elastic half-space; a
  non-commutative matrix {W}iener-{H}opf problem},
        date={1996},
        ISSN={0022-5096},
     journal={J. Mech. Phys. Solids},
      volume={44},
      number={12},
       pages={2125\ndash 2154},
  url={http://0-dx.doi.org.pugwash.lib.warwick.ac.uk/10.1016/S0022-5096(96)00064-6},
      review={\MR{1423542 (97i:73044)}},
}

\bib{Ab_noncom}{article}{
      author={Abrahams, I.~David},
       title={On the solution of {W}iener-{H}opf problems involving
  noncommutative matrix kernel decompositions},
        date={1997},
        ISSN={0036-1399},
     journal={SIAM J. Appl. Math.},
      volume={57},
      number={2},
       pages={541\ndash 567},
         url={http://dx.doi.org/10.1137/S0036139995287673},
      review={\MR{1438767 (97m:35041)}},
}

\bib{Pade}{article}{
      author={Abrahams, I.~David},
       title={The application of {P}ad\'e approximations to {W}iener-{H}opf
  factorization},
        date={2000},
        ISSN={0272-4960},
     journal={IMA J. Appl. Math.},
      volume={65},
      number={3},
       pages={257\ndash 281},
         url={http://dx.doi.org/10.1093/imamat/65.3.257},
      review={\MR{1806416 (2001k:45010)}},
}

\bib{Abraha_all_pade}{article}{
      author={Abrahams, I.~David},
      author={Davis, Anthony M.~J.},
      author={Llewellyn~Smith, Stefan~G.},
       title={Matrix {W}iener-{H}opf approximation for a partially clamped
  plate},
        date={2008},
        ISSN={0033-5614},
     journal={Quart. J. Mech. Appl. Math.},
      volume={61},
      number={2},
       pages={241\ndash 265},
         url={http://dx.doi.org/10.1093/qjmam/hbn004},
      review={\MR{2414435 (2009d:74047)}},
}

\bib{Mer}{article}{
      author={C{\^a}mara, M.~C.},
      author={Lebre, A.~B.},
      author={Speck, F.-O.},
       title={Meromorphic factorization, partial index estimates and
  elastodynamic diffraction problems},
        date={1992},
        ISSN={0025-584X},
     journal={Math. Nachr.},
      volume={157},
       pages={291\ndash 317},
         url={http://dx.doi.org/10.1002/mana.19921570124},
      review={\MR{1233066 (94j:47030)}},
}

\bib{Ann_rat}{article}{
      author={Conceicao, Ana~C.},
      author={Kravchenko, Viktor~G.},
      author={Pereira, Jose~C.},
       title={Rational functions factorization algorithm: a symbolic
  computation for the scalar and matrix cases},
        date={2012},
     journal={Proceedings of CSEI},
}

\bib{Crighton_matching}{article}{
      author={Crighton, D.~G.},
       title={Asymptotic factorization of {W}iener-{H}opf kernels},
        date={2001},
        ISSN={0165-2125},
     journal={Wave Motion},
      volume={33},
      number={1},
       pages={51\ndash 65},
  url={http://0-dx.doi.org.pugwash.lib.warwick.ac.uk/10.1016/S0165-2125(00)00063-9},
      review={\MR{1804459 (2001j:78012)}},
}

\bib{daniele2014wiener}{book}{
      author={Daniele, V.G.},
      author={Zich, R.S.},
       title={The {Wiener--Hopf} method in electromagnetics},
      series={Mario Boella Series on Electromagnetism in Information and
  Communication Series},
   publisher={Institution of Engineering and Technology},
        date={2014},
        ISBN={9781613530016},
         url={http://books.google.co.uk/books?id=lftHBAAAQBAJ},
}

\bib{Fokas_Painleve}{book}{
      author={Fokas, Athanassios~S.},
      author={Its, Alexander~R.},
      author={Kapaev, Andrei~A.},
      author={Novokshenov, Victor~Yu.},
       title={Painlev\'e transcendents},
      series={Mathematical Surveys and Monographs},
   publisher={American Mathematical Society, Providence, RI},
        date={2006},
      volume={128},
        ISBN={0-8218-3651-X},
         url={http://dx.doi.org/10.1090/surv/128},
        note={The Riemann-Hilbert approach},
      review={\MR{2264522 (2010e:33030)}},
}

\bib{Ga-Che}{book}{
      author={Gahov, F.~D.},
      author={{\v{C}}erski{\u\i}, Ju.~I.},
       title={Equations of convolution type (in russian)},
}

\bib{Constructive_review}{unpublished}{
      author={Gennady~Mishuris, Sergei~Rogosin},
       title={Constructive methods for factorization of matrix-functions},
        note={Submitted},
}

\bib{Spit}{incollection}{
      author={Gohberg, I.},
      author={Kaashoek, M.A.},
      author={Spitkovsky, I.M.},
       title={An overview of matrix factorization theory and operator
  applications},
    language={English},
        date={2003},
   booktitle={Factorization and integrable systems},
      editor={Gohberg, Israel},
      editor={Manojlovic, Nenad},
      editor={dos Santos, AntцЁnioFerreira},
      series={Operator Theory: Advances and Applications},
      volume={141},
   publisher={Birkhц╓user Basel},
       pages={1\ndash 102},
         url={http://dx.doi.org/10.1007/978-3-0348-8003-9_1},
}

\bib{existance}{article}{
      author={Gohberg, I.~C.},
      author={Kre{\u\i}n, M.~G.},
       title={Systems of integral equations on a half line with kernels
  depending on the difference of arguments},
        date={1960},
        ISSN={0065-9290},
     journal={Amer. Math. Soc. Transl. (2)},
      volume={14},
       pages={217\ndash 287},
      review={\MR{0113114 (22 \#3954)}},
}

\bib{Nigel_poles}{article}{
      author={Heaton, C.~J.},
      author={Peake, N.},
       title={Acoustic scattering in a duct with mean swirling flow},
        date={2005},
        ISSN={0022-1120},
     journal={J. Fluid Mech.},
      volume={540},
       pages={189\ndash 220},
         url={http://dx.doi.org/10.1017/S0022112005005719},
      review={\MR{2263126 (2007e:76231)}},
}

\bib{Khrap}{article}{
      author={Khrapkov, A.~A.},
       title={The first basic problem for a notch at the apex of an infinite
  wedge},
        date={1971},
        ISSN={0376-9429},
     journal={Internat. J. Fracture Mech.},
      volume={7},
       pages={373\ndash 382},
      review={\MR{0375896 (51 \#12085)}},
}

\bib{My1}{article}{
      author={Kisil, Anastasia~V.},
       title={A constructive method for an approximate solution to scalar
  {W}iener--{H}opf equations},
        date={2013},
        ISSN={1364-5021},
     journal={Proc. R. Soc. Lond. Ser. A Math. Phys. Eng. Sci.},
      volume={469},
      number={2154},
       pages={20120721, 17},
  url={http://0-dx.doi.org.pugwash.lib.warwick.ac.uk/10.1098/rspa.2012.0721},
      review={\MR{3035429}},
}

\bib{MyStrip}{article}{
      author={Kisil, Anastasia~V.},
       title={On the relationship between a strip wiener-hopf problem and a
  riemann-hilbert problem},
        date={2014},
     journal={Accepted, IMA J. Appl. Math},
}

\bib{history}{article}{
      author={Lawrie, Jane~B.},
      author={Abrahams, I.~David},
       title={A brief historical perspective of the {W}iener-{H}opf technique},
        date={2007},
        ISSN={0022-0833},
     journal={J. Engrg. Math.},
      volume={59},
      number={4},
       pages={351\ndash 358},
         url={http://dx.doi.org/10.1007/s10665-007-9195-x},
      review={\MR{2373797 (2009a:45002)}},
}

\bib{Spit_book}{book}{
      author={Litvinchuk, Georgii~S.},
      author={Spitkovskii, Ilia~M.},
       title={Factorization of measurable matrix functions},
      series={Operator Theory: Advances and Applications},
   publisher={Birkh\"auser Verlag, Basel},
        date={1987},
      volume={25},
        ISBN={3-7643-1883-X},
         url={http://dx.doi.org/10.1007/978-3-0348-6266-0},
        note={Translated from the Russian by Bernd Luderer, With a foreword by
  Bernd Silbermann},
      review={\MR{1015716 (90g:47030)}},
}

\bib{assym_2}{article}{
      author={{Mishuris}, G.},
      author={{Rogosin}, S.},
       title={{Factorization of a class of matrix-functions with stable partial
  indices}},
        date={2015-01},
     journal={ArXiv e-prints},
      eprint={1501.02692},
}

\bib{Mishuris-Rogosin}{article}{
      author={Mishuris, Gennady},
      author={Rogosin, Sergei},
       title={An asymptotic method of factorization of a class of matrix
  functions},
        date={2014},
     journal={Proceedings of the Royal Society A: Mathematical, Physical and
  Engineering Science},
      volume={470},
      number={2166},
  eprint={http://rspa.royalsocietypublishing.org/content/470/2166/20140109.full.pdf+html},
  url={http://rspa.royalsocietypublishing.org/content/470/2166/20140109.abstract},
}

\bib{Olver_steep}{article}{
      author={Olver, Sheehan},
      author={Trogdon, Thomas},
       title={Nonlinear steepest descent and numerical solution of
  {R}iemann-{H}ilbert problems},
        date={2014},
        ISSN={0010-3640},
     journal={Comm. Pure Appl. Math.},
      volume={67},
      number={8},
       pages={1353\ndash 1389},
         url={http://dx.doi.org/10.1002/cpa.21497},
      review={\MR{3225633}},
}

\bib{Speck_ind}{article}{
      author={Pr{\"o}ssdorf, S.},
      author={Speck, F.-O.},
       title={A factorisation procedure for two by two matrix functions on the
  circle with two rationally independent entries},
        date={1990},
        ISSN={0308-2105},
     journal={Proc. Roy. Soc. Edinburgh Sect. A},
      volume={115},
      number={1-2},
       pages={119\ndash 138},
         url={http://dx.doi.org/10.1017/S0308210500024616},
      review={\MR{1059649 (91e:15017b)}},
}

\bib{Application_DK}{article}{
      author={Rawlins, A.~D.},
       title={Two waveguide trifurcation problems},
        date={1997},
        ISSN={0305-0041},
     journal={Math. Proc. Cambridge Philos. Soc.},
      volume={121},
      number={3},
       pages={555\ndash 573},
         url={http://dx.doi.org/10.1017/S0305004196001296},
      review={\MR{1434661 (97i:76117)}},
}

\bib{Tref_old}{article}{
      author={Trefethen, Lloyd~N.},
       title={Rational chebyshev approximation on the unit disc},
        date={1981},
     journal={Numerische Mathematik},
}

\bib{Nigel_cyl}{article}{
      author={Veitch, B.},
      author={Peake, N.},
       title={Acoustic propagation and scattering in the exhaust flow from
  coaxial cylinders},
        date={2008},
        ISSN={0022-1120},
     journal={J. Fluid Mech.},
      volume={613},
       pages={275\ndash 307},
         url={http://dx.doi.org/10.1017/S0022112008003169},
      review={\MR{2462429 (2009j:76203)}},
}

\bib{Ab_Kh_n}{article}{
      author={Veitch, Benjamin~H.},
      author={Abrahams, I.~David},
       title={On the commutative factorization of {$n\times n$} matrix
  {W}iener-{H}opf kernels with distinct eigenvalues},
        date={2007},
        ISSN={1364-5021},
     journal={Proc. R. Soc. Lond. Ser. A Math. Phys. Eng. Sci.},
      volume={463},
      number={2078},
       pages={613\ndash 639},
         url={http://dx.doi.org/10.1098/rspa.2006.1780},
      review={\MR{2288836 (2007k:47023)}},
}

\end{biblist}
\end{bibdiv}

\end{document}